\documentclass[10pt]{amsart}
\usepackage{amscd,amssymb,float}

\newtheorem{theorem}{Theorem}[section]
\newtheorem{lemma}[theorem]{Lemma}

\newtheorem{corollary}[theorem]{Corollary}
\newtheorem{proposition}[theorem]{Proposition}

\renewcommand{\leq}{\leqslant}
\renewcommand{\geq}{\geqslant}

\theoremstyle{definition}
\newtheorem{definition}[theorem]{Definition}
\newtheorem{example}[theorem]{Example}

\theoremstyle{definition}

\numberwithin{equation}{section}

\numberwithin{equation}{section} \numberwithin{figure}{section}

\author{}
\address{}
\address{}
\email{}
\title{Wiener distribution  on holonomy groups}
\author{Yuguang Zhang}
\address{Xi'an Jiaotong-Liverpool University, Suzhou, China}
\email{Yuguang.Zhang@xjtlu.edu.cn}

\begin{document}
\begin{abstract}
  This paper proves a convergence theorem for the push-forward Wiener measures on holonomy groups  via stochastic parallel transports along convergent metric connections.
\end{abstract}

\maketitle
\footnotetext{Supported in part by grant NSFC-12531001.}

 \section{Introduction}
Connections are more fundamental   than curvatures. In  electromagnetism, the field strengths  are   interpreted as the  curvatures  of  $U(1)$-connections,  referred to as    potentials, and  charged particles interact with  the background    electromagnetic field through the    potential rather than  the  strength itself. The time-independent
Schrödinger equation has the form  $$\sum_i (\partial_i+\Gamma_i)^2\Psi + {\rm lower \ order \ terms} = 0,$$ where $\Psi$ is the  wave function,  and   $-\sqrt{-1}\Gamma$ is the  potential of  the  background    electromagnetic field.  The interaction terms are  $\Gamma_i\Psi$, which  involve only $\Gamma$ instead of the field  strength, i.e.,   essentially  the curvature $F=d\Gamma$.  The Aharonov-Bohm effect   illustrates  that
 when an  electron moves   along non-contractible loops $\gamma$  in a region with a magnetic field that has a  non-trivial potential $\Gamma\neq 0$ but a  vanishing strength  $d\Gamma=0$, i.e.,  a flat connection,     the wave function $\Psi$  acquires  an additional gauge invariant   phase factor $$\exp (- \int_\gamma \Gamma) \in U(1)$$ (See Section 10.5.3 of \cite{Na} for the details)  which is an element of the holonomy group of $\Gamma$.

Holonomy groups are invariants of connections  defined through  parallel transports.
  Let $M$ be a  connected  compact   manifold of dimension $n$,  and $g$ be a Riemannian metric on $M$. We consider
  a real  vector bundle $\pi: V \rightarrow M$  of rank $r$ on $M$,  and  a metric connection $\nabla$ on $V$ compatible with   a bundle  metric $h$,  i.e.,   $d h(\xi,\zeta)=h(\nabla\xi,\zeta)+h(\xi,\nabla\zeta)$ for any two sections $\xi$ and $\zeta$ of $V$.
  If $\gamma:[0,1]\rightarrow M$ is a  smooth curve, then the parallel transport $//_s(\gamma): V_{\gamma (0)}\rightarrow V_{\gamma (s)}$ along $\gamma$ is an isomorphism  given by $A_0 \mapsto A_s=//_s(\gamma)A_0$ where  $V_x=\pi^{-1}(x)$ is the fibre of $V$ over a point  $x\in M$, and
 $A_s $ is a  parallel section along $\gamma$, i.e.,  $\nabla_{\dot{\gamma}} A_s \equiv 0$ and $A_s\in V_{\gamma (s)}$.   If  $r=2$, and $V$ admits a complex structure  compatible  with the metric, then $V$ can be given  a complex Hermitian  line bundle structure, and  a metric connection $\nabla$ is a $U(1)$-connection.  Note that $U(1)=SO(2)$   and   $O(2)/SO(2)=\mathbb{Z}_2$.   For  complex line bundles,  if $\gamma$ and $\gamma'$ are curves homotopic to each other with $\gamma(0)=\gamma'(0)$ and $\gamma(1)=\gamma'(1)$, and  $\nabla$ is a $U(1)$-connection,  then  the parallel transports satisfy  $//_1(\gamma)=//_1(\gamma')\exp(-\int_{\iota(D)}F_\nabla) ,$  where $\iota: D\rightarrow M$ is a smooth  map from the unit  disc in $ \mathbb{C}$ to $M$ such that $\iota|_{\partial D}= \gamma \cdot \gamma'^{-1}$, and $F_\nabla$ is the curvature of $\nabla$ which is a pure imaginary 2-form.

 The holonomy group of $\nabla$ is defined as    $${\rm Hol}_x(\nabla)=\{//_1(\gamma)\in O(r)| \gamma\in \Omega_x^{\infty}\}, $$ where    $\Omega_x^{\infty}$ is  the set of   smooth loops based at $x\in M$, i.e.  curves $\gamma: [0,1]\rightarrow M$  with $\gamma(0)=\gamma(1)=x$.
 In the case of  the tangent bundle, i.e.,  $V=TM$,
  equipped with  the  Levi-Civita connection $\nabla_{h}$ of a Riemannian metric $h$,    $ {\rm Hol}_x(\nabla_h)$ is called the Riemannian holonomy group of $(M, h)$.   If $\nabla$ is a $U(1)$-connection on a complex line bundle $V$, then ${\rm Hol}_x(\nabla) \subseteq U(1)\subset O(2)$.

Elie Cartan introduced the notion of holonomy group in the 1920's, and it   has become an important research topic in differential geometry (cf. \cite{J}).  See \cite{BMKNZ} for applications of holonomies to  geometric phases   in physics.
Berger's famous  classification theorem (cf. \cite{Be, J}) asserts that if $M$ is simply connected, and $h$ is irreducible and non-symmetric, then the Riemannian holonomy group ${\rm Hol}_x(\nabla_h)$ can  only be equal to one of the following groups: $$SO(n), \  \  U(n/2), \ \   SU(n/2), \  \  Sp(n/4),  \  \  Sp(n/4)Sp(1),   \  \   G_2, \  \ Spin(7).$$   $(M,h)$ is  a K\"ahler (respectively, Calabi-Yau or HyperK\"ahler etc.) manifold  if ${\rm Hol}_x(\nabla_h)=U(n/2)$  (respectively,  $SU(n/2)$ or  $Sp(n/4)$ etc.).

One unsatisfactory   aspect of holonomy groups is their  behaviour  during the convergence  of connections.  To illustrate this, we consider a specific   example:   the trivial complex line bundle  $V\cong S^1\times\mathbb{C}$ on the circle $S^1$.  A $U(1)$-connection $\nabla=d+\Gamma$ is given by a pure imaginary 1-form $\Gamma$ on $S^1$, and is flat, i.e.,  $F_\nabla=d\Gamma=0$.  If $\nabla^t=d+t\Gamma$, $t\in [0,1]$,  then $\nabla^t$ converges to the trivial connection $\nabla^0=d$ when $t\rightarrow 0$.  The holonomy groups are ${\rm Hol}_x(\nabla^t)=\{e^{2\pi \nu t \theta \sqrt{-1}}\in U(1)| \nu\in \mathbb{Z}\}$ where  $2\pi  \theta \sqrt{-1}=-\int_{S^1}\Gamma$, and  ${\rm Hol}_x(\nabla^0)=\{1\in U(1)\}$.   If $t\theta$ is a rational number, then  ${\rm Hol}_x(\nabla^t)$ is finite, and if  otherwise,  ${\rm Hol}_x(\nabla^t)$ is dense in $U(1)$.  As $t\rightarrow 0$, ${\rm Hol}_x(\nabla^t)$ are  either finite groups or dense subgroups,  keep jumping,  and become denser and denser. Ultimately, we find that $${\rm Hol}_x(\nabla^t) \rightarrow U(1)$$ in the Hausdorff sense, and the limit is not  ${\rm Hol}_x(\nabla^0)$, even though   $\nabla^t \rightarrow\nabla^0$.

  This paper aims to associate the holonomy group with an additional  measure that  behaves  well during  the convergence of connections. Another motivation is to find  certain  measure theoretic   objects on  holonomy groups, which enable us to identify the limit holonomy along convergent sequences of   connections.

Heuristically, we  define a probability  measure $ \mu_x(\nabla)$ on $O(r)$ by
$$\int_K\mu_x(\nabla)=\mathcal{Z}^{-1}\int_{\{\gamma \in \Omega_x^{\infty}| //_1(\gamma) \in K \}}\exp ({-\frac{1}{4}\int_0^1| \dot{\gamma}(s)|^2_gds})\mathcal{D}(\gamma),$$  for any measurable    set $K\subseteq O(r)$, where $\mathcal{D}(\gamma)$ is the path integral volume form, which  unfortunately  does not exist, and $\mathcal{Z}$ is the normalisation    constant.  In  physics literature,   path integrals are usually  approximated  by finite dimensional  integrals.  If we consider a partition  $0=s_0< s_1<\cdots < s_m=1$ with $s_i=\frac{i}{m}$  and $m=2^\nu$, $\nu \in \mathbb{N}$,  then for any $(x_1, \cdots, x_{m-1})\in M^{m-1}$, we choose a piecewise curve $\gamma$ such that $ \gamma(s_i)=x_i$, and  $\gamma(s)$ is a minimal  geodesic for $s\in [s_i,s_{i+1}]$ where $x=  \gamma(s_0)=\gamma(s_m)$. This  defines a map $\tilde{\Phi}^m: M^{m-1} \rightarrow \Omega_x^\infty$. Heuristically, we define the approximation measures    $\mu^m_x(\nabla)$     by
$$\int_K\mu^m_x(\nabla) \sim\frac{1}{(4\pi (1/m))^{n(m-1)/2}} \int_{(//_1\circ\tilde{ \Phi}^m)^{-1}(K) } e^{-\sum\limits_{i=0}^{m-1}\frac{d_g^2(x_{i+1}, x_i)}{4(1/m)}}\prod_{i=1}^{m-1} dv_g(x_i) , $$ and we  expect  that $\mu^m_x(\nabla)$ converges to $\mu_x(\nabla)$ in the distribution sense when $m\rightarrow\infty$, where $d_g$ denotes the distance function of $g$.
Here  $a\sim b$ means that $|a-b|\rightarrow 0$ when $m\rightarrow\infty$.

 In
   Section 2, we define      $\mu_x(\nabla)$ rigorously  as the push-forward   Wiener measure on the space of loops via  the stochastic  holonomy map  (Definition \ref{def}). This measure  can be  viewed as the `shadow' of the   Wiener measure on the holonomy group.  Finite dimensional  approximations of the  Wiener measure have been studied from various  perspectives, for example  \cite{AD,DWM} etc..
 Section 3  defines    $\mu^m_x(\nabla)$ rigorously  (Definition \ref{def+}), and proves  an  approximation theorem (Theorem \ref{thm+2}).  Furthermore,  we conjecture\footnote{It might be a direct consequence of Strock-Varadhan's support theorem (See \cite{IN,MSS}). We leave it as an exercise  to interested readers.} that the support of $ \mu_x(\nabla)$ is the closure of the holonomy group, i.e.,    $${\rm supp} (\mu_x(\nabla))= \overline{ {\rm Hol}_x(\nabla)} $$ in $O(r)$, and prove a weaker result  in Section 3, which is   $$\overline{ \bigcup_{ m\in \mathbb{N}}{\rm supp} (\mu_x^m(\nabla))} =\overline{ {\rm Hol}_x(\nabla)}.$$

  Our  main result is the following theorem.

  \begin{theorem}\label{main}   Let $(M,g)$ be a smooth connected  compact Riemannian  $n$-dimensional  manifold, and   $x\in M$. Consider    a real  vector bundle  $\pi: V \rightarrow M$ of rank $r$ on $M$, which is equipped with a metric connection $\nabla$ on $V$ that is   compatible with a bundle metric $h$. Suppose that    $\nabla^k$ is a family of connections compatible with  metrics  $h^k$   on $V$  such that $$ \|\nabla^k-\nabla\|_{C^1(g\otimes h)}\rightarrow 0,$$   and $h^k|_x \rightarrow  h|_x$  in $V_x^* \otimes V_x^*$,
  when $k\rightarrow \infty$.
Then  $ \mu_x(\nabla^k)$ converges to $ \mu_x(\nabla)$ in the distribution sense   when $k\rightarrow \infty$, i.e.,  for any continuous function $\psi$ on $O(r)$, $$ \lim_{k\rightarrow \infty}\int_{O(r)}\psi \mu_x(\nabla^k)= \int_{O(r)}\psi \mu_x(\nabla).$$
\end{theorem}

Here the $C^1$-norm $  \|\cdot\|_{C^1(g\otimes h)}$ is defined as $\|\zeta\|_{C^1(g\otimes h)}=\sup\limits_{y\in M}( |\zeta|_{g\otimes h}+ |\nabla' \zeta|_{g\otimes h})(y)$ for any differentiable   section $\zeta$ of $T^*M\otimes {\rm Hom}(V,V)$, where the metric $|\cdot|_{g\otimes h}$ (resp. the connection $ \nabla'$) is induced by the Riemannian metric $g$ and the bundle metric $h$ (resp. $\nabla$ and the Levi-Civita connection of $g$).  

 We can  apply this  theorem to
 many natural scenarios where    a sequence of Riemannian metrics $h^k$ converges  smoothly  to a Riemannian metric $h$  such that the holonomy group  of $h$ is a proper subgroup of the holonomy groups   of $h^k$.
In \cite{Yau},  Yau   uses the continuity method to prove the existence of Ricci-flat Kähler-Einstein metrics on compact Kähler manifolds with vanishing first Chern class, i.e.,  $c_1=0$. The proof involves a family of Kähler metrics $h^t$, $t\in [0,1]$,  converging  smoothly  to a Ricci-flat Kähler-Einstein metric $h$ when $t\rightarrow 1$. When $t<1$, $h^t$ is not  Ricci-flat.  The holonomy groups are ${\rm Hol}_x(\nabla_{h^t})=U(n/2)$  for $ t<1$,  and  $ {\rm Hol}_x(\nabla_{h^1})=SU(n/2)$.  Before reaching  the limit, all the holonomy groups are the same $U(n/2) $; however, at  the limit, the holonomy collapses suddenly  to a  subgroup $H$ that is  isomorphic to  $SU(n/2)$. However,   if we consider the measures  $ \mu_x(\nabla_{h^t})$, then Theorem \ref{main} asserts that  $ \mu_x(\nabla_{h^{t_k}})\rightharpoonup \mu_x(\nabla_{h^1})$ for  sequences $t_k \rightarrow 1$, which indicates that  $ \mu_x(\nabla_{h^t})$ concentrates near $H \cong SU(n/2)$. We may regard $h^t$ as `almost Calabi-Yau' metrics if $t$ is sufficiently close to 1.

The same phenomenon also occurs in  Cao's Ricci-flow proof of the Calabi-Yau theorem (cf. \cite{Cao}), as well as in  Joyce's   constructions of compact $G_2$ and $Spin(7)$ manifolds (cf. \cite{J}) etc..    In the work \cite{DWW} of Dai-Wang-Wei,  Riemannian metrics with special holonomy groups $SU(n/2)$, $G_2$, and $Spin(7)$ are local maxima of  the Yamabe functional. Some min-max sequences  of metrics exhibit  the holonomy groups  $SO(n)$, and the limit Riemannian  metric has the smaller  holonomies, e.g. $ SU(n/2)$.
 \cite{FZ,FZZ1,FZZ2} provide   examples where the holonomy group $SO(4)$ collapses  to $U(2)$ at the infinite limit of long time  solutions of Ricci-flow.  We can  apply  Theorem \ref{main}   to all the above cases.

   The following  result shows that the limit  of  approximation measures  indicates   the limit holonomy.

  \begin{theorem}\label{thm3+}
       Let $M$, $V$, $g$, $h^k$, $h$,  $\nabla^k$, and $\nabla$ be the same as in Theorem \ref{main}.
          \begin{itemize}
\item[i)] For any $m$,  $ \mu_x^m(\nabla^k)$  converges to $ \mu_x^m(\nabla)$ in the distribution sense when $k \rightarrow\infty$.
  \item[ii)]If $H\subseteq O(r)$ is a closed subgroup, and for any open subset $K \subset O(r)\backslash H$ with the closure  $\overline{K} \subset O(r)\backslash H$,   $$\lim_{k\rightarrow\infty}\int_{ K} \mu_x^m(\nabla^k) =0,   $$ for all $m>0$,  then   ${\rm Hol}_x(\nabla)\subseteq H$.     \end{itemize}
    \end{theorem}

Finally,   we apply Theorem \ref{main} to  flat $U(1)$-connections on Lagrangian submanifolds in   symplectic manifolds.
    Let $N$ be a compact  manifold of dimension $2n$ that admits  a symplectic  form $\omega$  representing the first Chern class $c_1(L)$ of a complex line bundle $L$.
   Hence there is a  $U(1)$-connection $\nabla$ on $L$  such that the curvature $F_\nabla=-2\pi \sqrt{-1}\omega$ by the Chern-Weil theory.  If $ M$ is a Lagrangian submanifold of $N$, then  the restriction of $\nabla$ to $M$  yields a flat $U(1)$-connection on $V=L|_M$.  $M$ is called   a Bohr-Sommerfeld Lagrangian submanifold if $\nabla|_M$ is gauge equivalent to  the trivial connection, i.e.,
    $\nabla|_M\cong d$  (cf. \cite{GS}).  Suppose that  there is  a smooth  family of Lagrangian submanifolds $M_t$, $t\in [0,1)$, such that $M_0$ is Bohr-Sommerfeld Lagrangian, and all  $M_t$ are  diffeomorphic   to $M$.  We view  $\nabla^t=\nabla|_{M_t}$ as connections  on the same manifold  $M \cong M_t$, $t\in [0,1)$, i.e.,  a family of flat $U(1)$-connections,   which satisfies  that $\nabla^t\rightarrow \nabla^0$ smoothly  when $t\rightarrow0$.  For a sequence $t_k \rightarrow 0$,  $\mu_x(\nabla^{t_k})\rightharpoonup \mu_x(\nabla^0) $ on $U(1)$ by Theorem \ref{main}.  Conversely, we obtain  the following  corollary.

     \begin{corollary}\label{pro-lag}    Let $(N, \omega)$ be an integral symplectic manifold, and let  $\nabla$ be  a $U(1)$-connection  with   curvature $F_\nabla=-2\pi \sqrt{-1}\omega$. If   $M_t$, $t\in [0,1)$, is a  family of Lagrangian submanifolds diffeomorphic to $M$  such that for any $\epsilon>0$,     $$\lim_{t \rightarrow 0}\int_{ \{e^{2\pi\sqrt{-1}\theta}\in U(1)| \epsilon<\theta<1-\epsilon\}} \mu_x(\nabla|_{M_t})= 0,$$ where $\mu_x(\nabla|_{M_t})$ are calculated via a Riemannian metric on $M$ and an  $x\in M$, then $M_0$ is a Bohr-Sommerfeld Lagrangian submanifold.
 \end{corollary}

  In Section 4,   we  prove Theorem \ref{main}, Theorem \ref{thm3+}, and Corollary  \ref{pro-lag}, and
conduct a  detailed computation  for the case of    flat $U(1)$-connections, which provides  an explicit example   of Theorem \ref{main}.

{\bf Acknowledgments. } The author thanks the referees for pointing out an error.

\section{Preliminaries}
In this section,
we review some background materials about stochastic calculus on manifolds, mainly taken  from  the paper of B\"ar-Pf\"affle \cite{BP},  the  books \cite{Em,Evans}, and Section 3 in \cite{Dr1},  and  provide  the  rigorous definition of the measure   $ \mu_x(\nabla)$.   We also consult  the  articles \cite{EL, L}.

\subsection{Wiener measure}
Let $(M,g)$ be a compact Riemannian $n$-manifold, and let
   $p_s(x,y)$ denote   the heat kernel of the  Laplace operator $\Delta_g$, i.e.,  $p_s(x,y)$ solves the initial problem for  the heat equation   $$\frac{d p_s(x, y)}{d s}=\Delta_g p_s(x, y), \  \  \  {\rm and } \  \ \ p_s(x, y)dv_g\rightharpoonup \delta_y(x),$$  when $s\rightarrow 0$, where $\delta_y(x)$ is the Dirac delta function, and $dv_g$ is the volume measure  associated with  $g$.
    The heat kernel exhibits the following properties (cf. Section 3 in \cite{BP}):      \begin{itemize}
\item[i)] $p_s(x,y)>0$ and $p_s(x,y)=p_s(y,x)$.
\item[ii)] Stochastic completeness   $$\int_M p_s(x, y) dv_g(x)=1 .$$
\item[iii)] $$ p_{s_1+s_2}(x,y)=\int_M p_{s_1}(x,z)p_{s_2}(z,y)dv_g(z). $$
\item[iv)]   Asymptotic  expansion  $$ p_s(x, y) \sim \frac{1}{(4\pi s)^{n/2}}e^{-\frac{d_g^2(x,y)}{4s}}, $$ as   $s\ll 1$, where $d_g$ is the distance function of $g$.
 \end{itemize}

Let   $\Omega_x$ be   the loop space based  at a point  $x\in M$, i.e.,
  $$\Omega_x= \{\gamma| {\rm continuous \ curves} \  \gamma: [0,1]\rightarrow M \ {\rm with }\  x=\gamma(0)=\gamma(1)\}, $$ equipped with the $C^0$-topology.
  There is a unique Borel  measure  $ \mathcal{W}_x(g)$ called the   conditional Wiener measure (or pinned Wiener measure) on  $\Omega_x$ by  Corollary 2.20 in \cite{BP} satisfying  that for  a    partition  $0<s_1<\cdots <s_{m}<1$ and any open subsets $U_1, \cdots, U_{m}\subset M$,
   \begin{eqnarray*} & &\int_{\{\gamma\in\Omega_x|\gamma(s_1)\in U_1, \cdots, \gamma(s_{m})\in U_{m}\} }\mathcal{W}_x(g)\\ & = & \frac{1}{p_1(x,x)} \int_{U_1\times\cdots\times U_{m}}p_{1-s_{m}}(x, x_{m})
   \prod_{i=2}^{m} p_{s_i-s_{i-1}}(x_{i}, x_{i-1})p_{s_1}(x_1, x) \prod_{i=1}^{m}dv_g(x_i).\end{eqnarray*} Note that our definition differs slightly from that in    \cite{BP} as we further divide  the   conditional Wiener measure defined  in  \cite{BP} by an additional factor of   $p_1(x,x)$.  The measure
    $ \mathcal{W}_x(g)$ is a probability measure, i.e.,  $\int_{\Omega_s}\mathcal{W}_x(g)=1 $.  Furthermore,  $ \mathcal{W}_x(g)$ is supported in the space $C^\beta(\Omega_x)\subset \Omega_x$ of  $\beta $-H\"older  curves   for  $0\leq \beta < \frac{1}{2}$.    Unfortunately, the smooth loop space  $\Omega_x^{\infty}$ has  $ \mathcal{W}_x(g)$-measure zero.

  Note that if $M$ is not simply connected, i.e.,  the fundamental group $\pi_1(M)\neq  \{1\}$, then $\Omega_x$ has many connected components, and $$\Omega_x=\coprod_{\nu \in \pi_1(M)} \Omega_x(\nu) $$ where $\Omega_x(\nu) $ denotes the space of loops representing the class $\nu \in \pi_1(M)$.  If $f: \tilde{M}\rightarrow M$ is  the universal covering of $M$, and $f(y)=x$ for a $y\in \tilde{M}$,  then $\pi_1(M)$ acts on $\tilde{M}$, and  $x=f(\nu\cdot y)$ for any $\nu\in \pi_1(M)$. Let  $\tilde{g}=f^*g$,  $\tilde{p}_s(y_1, y_2)$ be the heat kernel of $\tilde{g}$, and $\tilde{\Omega}^y_{\nu\cdot y}$ be the space of continuous  curves $\tilde{\gamma}:[0,1]\rightarrow \tilde{M}$ with $\tilde{\gamma}(0)=y$ and $\tilde{\gamma}(1)=\nu\cdot y$.  The covering map $f$  induces a bijective continuous map from $\tilde{\Omega}^y_{\nu\cdot y}$ to $\Omega_x(\nu) $. Theorem 4.3 in \cite{BP} shows that $f_*\tilde{p}_1(y, \nu\cdot y) \mathcal{W}^y_{\nu\cdot y}(\tilde{g})= p_1(x,x)\mathcal{W}_x(g)|_{\Omega_x(\nu)}$ where $\mathcal{W}^y_{\nu\cdot y}(\tilde{g})$ denotes the conditional  Wiener measure on $\tilde{\Omega}^y_{\nu\cdot y}$ (cf. Corollary 2.20 in \cite{BP}) that is normalised such that  $ \int_{\tilde{\Omega}^y_{\nu\cdot y}}\mathcal{W}^y_{\nu\cdot y}(\tilde{g})=1$.  Therefore,  \begin{equation}\label{eq-component}\int_{\Omega_x(\nu)} \mathcal{W}_x(g)= \frac{\tilde{p}_1(y, \nu\cdot y) }{ p_1(x,x)}>0.  \end{equation}

   If $\mathcal{B}$ denotes the Borel $\sigma$-algebra of $\Omega_x$, then we define $\mathbb{P}_g: \mathcal{B} \rightarrow \mathbb{R}$,    by  $ B \mapsto \int_B \mathcal{W}_x(g),$ and $\mathbb{P}_g(\Omega_x)=1 $. Thus
        $(\Omega_x,  \mathcal{B},  \mathbb{P}_g)$ is a probability space.  Given $B\in \mathcal{B}$ with $ \mathbb{P}_g(B)>0$, the conditional probability is defined as $ \mathbb{P}_g(C|B)=\frac{\mathbb{P}_g(C\cap B)}{\mathbb{P}_g(B)}$ for any $C\in \mathcal{B}$.   If $  {\bf 1}_B$ is the characteristic   function of $B$, then $\mathbb{P}_g(C|B)= \int_C \frac{ {\bf 1}_B}{\mathbb{P}_g(B)} \mathcal{W}_x(g)$, i.e., the associated measure of $\mathbb{P}_g(\cdot|B)$ is $\frac{ {\bf 1}_B}{\mathbb{P}_g(B)} \mathcal{W}_x(g)$.
          For a  measurable function $\psi$ on $\Omega_x$,
   the  expectation of $\psi$ is denoted as   $$\mathbb{E}(\psi)=\int_{\Omega_x}\psi \mathcal{W}_x(g).$$

On the probability space $(\Omega_x,  \mathcal{B}, \{ \mathcal{B}_s\},  \mathbb{P}_g)$, the coordinate process is a family of measurable maps  $X_s: \Omega_x \rightarrow M$, $0\leq s \leq 1$, given by   $\gamma\mapsto X_s(\gamma)=\gamma(s)$ for any $\gamma \in \Omega_x$, and  is   called a  Brownian bridge from $x$ to $x$, where  $\{ \mathcal{B}_s\} $ is a filtration satisfying the `usual hypothesis' (cf. Section 3 in \cite{Dr1}).   Theorem 2.3 of  \cite{Dr2} asserts that $X_s$ is an $M$-valued semimartingale with respect  to $(\Omega_x,  \mathcal{B}, \{ \mathcal{B}_s\},  \mathbb{P}_g)$, which enable us  to quote  directly many results in Emery's book \cite{Em} by   extending   the time of  $X_s$ from $[0,1]$ to $[0,\infty)$ via  $X_s\equiv x$ for all $s\in [1, \infty)$.
If $ T: \Omega_x \rightarrow [0,1]$ is a stopping time, then the new process $X_s^{|T}$,  called $X_s$ stopped at $T$,  is defined as $X^{|T}_s(\gamma)=X_s(\gamma)$ if $s< T(\gamma)$, and $X^{|T}_{s}(\gamma)=X_{T(\gamma)}(\gamma)$ if $ T(\gamma)\leq s$.
 For two stopping times $S$ and $ T$ with $S\leq T$, the stochastic interval is   defined as $[[S, T]]=\{(\gamma, s)\in \Omega_x \times [0,1]| S(\gamma)\leq s \leq T(\gamma)\}$ (cf. (1.2) in Chapter 1 of \cite{Em}).

If $\alpha$  is a smooth 1-form on $M$ and $\gamma$ is a smooth curve, then the integral $\int_\gamma \alpha =\int_0^1 \alpha \langle  \dot{\gamma}(s)\rangle ds$ is well defined, where $\alpha \langle  \cdot \rangle$ denotes the pairing  of $\alpha$ with a vector.
The Stratonovich stochastic integral  of $\alpha$ along $X_s$ is a well-defined process $\int_0^s\alpha \langle \delta X_t\rangle$ (cf. Definition 7.3 in \cite{Em} and Definition 3.2 in \cite{Dr1}), i.e.,   a family of measurable functions on $(\Omega_x,  \mathcal{B}, \{ \mathcal{B}_s\},   \mathbb{P}_g)$, which can be approximated by the integrating  $\alpha$ along piecewise smooth curves as follows.
   By Proposition 7.13 in \cite{Em},  there is a measurable map $I: M\times M\times [0,1]\rightarrow M $, called an interpolation rule, satisfying the following.
     \begin{itemize}
\item[i)]  For any $(x,y)\in M$, the map $s \mapsto I(x, y,s)$ is a smooth curve,
\item[ii)] there is a neighbourhood $\Lambda$ of the diagonal of $M\times M$ such that  $I$ is smooth on $\Lambda\times [0,1]$,
\item[iii)] and for each $(x,y)\in \Lambda$, the curve $s \mapsto I(x, y,s)$ is a geodesic.
  \end{itemize}
  For the  partition  $\{s_i=i/m|i=0, 1, \cdots, m\}$ of $[0,1]$, we define  a process $X_s^m:  \Omega_x\times [0,1] \rightarrow M$ by  $X_s^m(\gamma) = I(X_{s_i}(\gamma), X_{s_{i+1}}(\gamma), m(s-s_i)), $ for any $s\in [s_i, s_{i+1}]$, $ i=0,\cdots, m-1$,  and any $\gamma \in \Omega_x$,  referred as  the interpolated process. Consequently   $X_s^m(\gamma)$ is piecewise smooth.  We regard $X_s^m$ as a measurable map $X^m:  \Omega_x \rightarrow \Omega_x^\infty$ given by $\gamma\mapsto X^m(\gamma)$ with $ X^m(\gamma)( s)= X^m_s(\gamma)$.
   Theorem 7.14 in \cite{Em} shows that $$  \mathbb{E}\big(\sup_{s\in [0,1]}\big|\int_0^s\alpha \langle \delta X_t\rangle-\int_0^s\alpha \langle \dot{X}_{t}^m\rangle dt\big|\big) \rightarrow 0,$$ when $m \rightarrow \infty$, i.e.,   $\int_0^s\alpha \langle \dot{X}_{t}^m\rangle dt$ converges to $\int_0^s\alpha \langle \delta X_t\rangle$ in probability.  Thus
   there is a full measure subset $ \Omega_x''\subset \Omega_x$, i.e. $\int_{\Omega_x''} \mathcal{W}_x(g)=1 $,  such that if $\gamma \in \Omega_x''$,  $\int_0^1\alpha \langle \delta X_s\rangle (\gamma) $ is a well-defined number, and $\int_0^1\alpha \langle \dot{X}_{s}^m(\gamma)\rangle ds \rightarrow \int_0^1\alpha \langle \delta X_s \rangle (\gamma) $ subsequently  when $m\rightarrow \infty$, i.e.,  $\int_0^1\alpha \langle \dot{X}_{s}^m\rangle ds\rightarrow\int_0^1\alpha \langle \delta X_s\rangle$ almost surely (cf. Theorem 1.32 in \cite{ELE}) by passing to  subsequences.  Certainly, we can assume that $\Omega_x''\subseteq C^\beta(\Omega_x)$ for any $0\leq \beta < \frac{1}{2}$.

   We remark that $ \int_0^1\alpha \langle \delta X_s  \rangle (\gamma)$ can be computed    by using smooth objects, which might hold  independent interest even though   not being  used in the proof of the main theorem.
   If $\gamma^o$ is a smooth curve homotopy to $\gamma \in \Omega_x''$, then  we apply the Stokes' theorem to a family of maps $\iota_m: D \rightarrow M$  such that $\iota_m|_{\partial D}=X_s^m(\gamma)\cdotp (\gamma^o)^{-1}$ where $D$ is the standard unit   disc in $\mathbb{C}$, i.e., $\int_{\iota_m(\partial D)} \alpha=\int_{\iota_m(D)}d \alpha$.   This leads us to  the following result.

    \begin{proposition}[Stokes' formula]\label{le--}  Suppose that $\alpha$ is a smooth 1-form on $M$ and  $\gamma\in \Omega''_x$. If $\gamma^o\in \Omega_x^\infty$ is  homotopy to $\gamma$, and
    $\iota: D \rightarrow M$ is a map from the unit disc in $ \mathbb{C}$ to $M$   such that $\iota$ is smooth in the interior of $D$, and   $\iota|_{\partial D}=\gamma\cdotp (\gamma^o)^{-1}$, then $$\int_0^1\alpha \langle \delta X_s \rangle (\gamma)= \int_{\gamma^o}\alpha+\int_{\iota(D)}d \alpha.  $$
 \end{proposition}

\subsection{Stochastic parallel transports}
Let $\pi: V \rightarrow M$ be
  a real  vector bundle   of rank $r$ on $M$ equipped with  a bundle  metric $h$,  and  a metric connection $\nabla$ on $V$.  If $\gamma$ is a  smooth curve in $M$, then a parallel section $A$ of $V$ along $\gamma$ satisfies the equation $\nabla A \langle \dot{\gamma} \rangle =0$.  Under local coordinates  $x^1, \cdots, x^n$ on an open subset  $U$  and  a trivialisation  of $V$,  $V|_{U} \cong U\times \mathbb{R}^r$, $\nabla= d+ \sum\limits_{ i=1}^{n}\Gamma_i dx^i$ and the equation reads $$\frac{d A_s}{d s} +\sum\limits_{ i=1}^n\Gamma_i(\gamma (s)) \dot{\gamma}^i(s)A_s=0$$ where $A_s=A(\gamma (s))$ is regarded as  a curve on $\mathbb{R}^r$, and $\Gamma_i: U \rightarrow {\rm Hom}(\mathbb{R}^r,\mathbb{R}^r)$ are smooth.  If  $\nabla^k=d+\Gamma^k$ is a family of connections such that $\Gamma^k \rightarrow\Gamma$ in the $C^0$-sense, and $A_s^k$ is a parallel section under $\nabla^k$ along $\gamma$ with the   initial values  $A^k_0\rightarrow A_0$ in $V_x$,    then $A_s^k$ converges to $A_s$, $s\in [0,1]$, when $k\rightarrow\infty$ by the stability of ordinary differential  equations with respect to variations   of coefficients.

   A parallel section $A$ of $V$ along the process  $X_s$ is a measurable map $A: \Omega_x\times [0,1]\rightarrow V$  satisfying the  stochastic differential   equation  \begin{equation}\label{eq-1}\nabla A_s \langle \delta X_s\rangle  =0,  \  \  \  \  \pi(A_0)=x,  \end{equation}
   in the Stratonovich  sense, and  $\gamma(s)=X_s(\gamma)=\pi (A_s(\gamma))$,  where $A_s(\gamma)=A(\gamma,s)$. The equation (\ref{eq-1}) can be expressed  in coordinates as follows.

      If
$ \mathcal{U}=\{U \}$
   is  a coordinate open covering      of $M$, then we apply  Lemma 3.5 in \cite{Em} to the process $Y_s=X_s $ if $s\in[0,1]$, and $Y_s \equiv x$ if $s\in [1, \infty)$. Hence
    there is  a sequence of stopping times $0=T_0 \leq T_1 \leq \cdots  \leq T_m \leq \cdots$ such that $\sup\limits_{m} T_m=\infty $,  and  for any $j$ and  any  $(\gamma, s)\in \mathcal{I}_j= [[S_j, S_{j+1}]]\cap \{(\gamma, s) \in
   \Omega_x\times [0,1] |S_j(\gamma) < S_{j+1}(\gamma) \}$, there is a   $U\in \mathcal{U}$ with $X_s(\gamma)=\gamma(s)\in U$ where $S_i=\min \{1, T_i\}$.
 On each   $U$,  the equation (\ref{eq-1}) reads   \begin{equation}\label{eq-1+1}dA_s+\sum_{ i=1}^n \Gamma_{ i}(X_{s})A_{s} \delta  X_{s}^i=0,   \end{equation} when $X_s(\gamma)=\gamma(s)\in U$, i.e.,  $(\gamma, s)\in  \mathcal{I}_j$,  or equivalently,  \begin{equation}\label{eq0}A_{s_1}=A_{s_0} - \sum_{ i=1}^n \int_{s_0}^{s_1}\Gamma_{ i}(X_{s})A_{s} \delta  X_{s}^i, \end{equation}  in the Stratonovich sense  (cf.  (8.12) in \cite{Em}) under the coordinates $x^1, \cdots, x^n$ on $U$ and the trivialisation $V|_U\cong U\times\mathbb{C}^r$.
 Now, suppose that  $x'^1, \cdots, x'^n$ are different  coordinates on $U'$ and $V|_{U'} \cong U'\times \mathbb{R}^r $ is another   trivialisation  of $V$, and the transition function is  $q: U \cap U' \rightarrow {\rm Hom} (\mathbb{R}^r, \mathbb{R}^r)$. Then $A_s'(\gamma)=q(\gamma(s))A_s(\gamma)$, $\sum\limits_{j}\Gamma'_jdx'^j=qdq^{-1}+q\sum\limits_{i}\Gamma_idx^i q^{-1}$,  $\delta  X_{s}'^j  =\sum\limits_{i} \frac{\partial x'^j}{\partial x^i} \delta  X_{s}^i$, and
  $$dA_s'+\sum_{j=1}^n \Gamma_{j}'(X_{s})A_{s}' \delta  X_{s}'^j  =   q(dA_s+\sum_{ i=1}^n \Gamma_{ i}(X_{s})A_{s} \delta  X_{s}^i)=0.$$

  Now we convert the equation  (\ref{eq-1+1})  into a  stochastic differential equation  of Itô sense. Denote $B(A_s, X_s)=-[\Gamma_{ 1}(X_{s})A_s,\cdots, \Gamma_{ n}(X_{s})A_s ]$ for  all $s\in[0,1]$, $B_i^\alpha=
  -\sum\limits_{\beta=1}^r\Gamma_{ \beta i}^\alpha(X_{s})a^\beta$, where $A_s=[a^1, \cdots, a^r]^T $,  and $B_i(A_s, X_s)=[B_i^1, \cdots, B_i^r]^T$, $i=1, \cdots, n$.  We rewrite  (\ref{eq-1+1}) as $$d\begin{bmatrix}
A_s \\
X_s
\end{bmatrix}= \begin{bmatrix}
B(A_s, X_s) \\
I_{n\times n}
\end{bmatrix}\delta   X_{s}, $$ which is equivalent to the  Itô stochastic differential equation
$$d\begin{bmatrix}
A_s \\
X_s
\end{bmatrix}= \begin{bmatrix}
b(A_s,s) \\
0
\end{bmatrix}ds+  \begin{bmatrix}
B(A_s, X_s) \\
I_{n\times n}
\end{bmatrix} \hat{d}  X_{s}, $$ where $\hat{d}$ denotes the Itô differential, and $I_{n\times n}$ is the identity matrix.  Section 6.5.6 in \cite{Evans} shows that if  $b(A_s,s)=[b^1, \cdots, b^r]^T$, then  $$b^\alpha=\frac{1}{2}\sum_{ i=1}^n(
\sum_{ \beta=1}^r\frac{\partial B^\alpha_i}{\partial a^\beta} B^\beta_i+\frac{\partial B^\alpha_i}{\partial x^i} )= \frac{1}{2}\sum_{ i=1}^n\sum_{ \beta=1}^r(\sum_{ \gamma=1}^r\Gamma_{ \beta i}^\alpha(X_{s})\Gamma_{ \gamma i}^\beta(X_{s})a^\gamma- \frac{\partial \Gamma_{ \beta i}^\alpha}{\partial x^i} (X_{s})a^\beta), $$ $\alpha=1, \cdots, r$.   Therefore, we obtain \begin{equation}\label{eq0+} dA_s=\frac{1}{2}\sum_{ i=1}^n  (\Gamma_{ i}(X_{s})\Gamma_{ i}(X_{s})-\frac{\partial \Gamma_{ i}}{\partial x^i} (X_{s}))A_{s}ds- \sum_{ i=1}^n \Gamma_{ i}(X_{s})A_{s} \hat{d}  X_{s}^i\end{equation} in the Itô sense.

  If $M$ is a torus and $V$ is a trivial bundle, i.e. $M=\mathbb{R}^n/\mathbb{Z}^n$ and $V\cong \mathbb{R}^n/\mathbb{Z}^n \times\mathbb{R}^r$,  then $\nabla=d+\sum\limits_{i}\Gamma_idx^i$, where $x^1,\cdots , x^n$ are angle coordinates, i.e.,  $x^i\in \mathbb{R}/\mathbb{Z}$. The equations (\ref{eq0}) and (\ref{eq0+}) are globally  defined, and   the pull-back equations on the universal covering  $\mathbb{R}^n  $ are respectively   $\mathbb{Z}^n$-periodic   Stratonovich and Itô equations driven by the coordinate  process on the probability space $(\coprod\limits_{\nu \in \mathbb{Z}^n}\Omega_{\mathbb{R}^n,\nu}^0, \mathcal{B}_{\mathbb{R}^n},  \mathbb{P}_{\mathbb{R}^n}) $, where $\Omega_{\mathbb{R}^n,\nu}^0$ is the space of continuous curves in $\mathbb{R}^n$ connecting $0$ and $\nu$, $\mathcal{B}_{\mathbb{R}^n}$ is the Borel $\sigma$-algebra, and $\mathbb{P}_{\mathbb{R}^n}$ is defined  by the conditional Wiener measure on $\Omega_{\mathbb{R}^n,\nu}^0$ by  Theorem 4.3 in \cite{BP}.
   Therefore,  we apply the theorems of existence, uniqueness,  and stability for solutions of  stochastic differential equations (cf.  Theorem 7 and Theorem 15 in Chapter V of \cite{P}, or
    Section 5.2.3-5.3 in \cite{Evans}) to   (\ref{eq0+}), and obtain the following conclusions.
    \begin{itemize}
\item[i)] For any $A_0 \in V_x$, there exists a unique  solution  $A$ satisfying (\ref{eq0+}) with initial value $A_0$.
  \item[ii)]   Suppose that $\nabla^k=d+\Gamma^k$ is a sequence of metric connections such that $\Gamma^k$ converges smoothly  to $\Gamma$ when $k\rightarrow\infty$.  If $A_s^k$ are  parallel section with respect to $\nabla^k$ with the  initial values    $A^k_0\rightarrow A_0$ in $V_x$, then $$\mathbb{E}\big(\sup_{s\in [0,1]} |A_s^k-A_s | \big)\rightarrow 0,$$ when $k\rightarrow\infty$, i.e.,  $A_s^k$ converges to $A_s$ in probability.
   \end{itemize}

In the general case,   the same arguments  as in the proofs of Proposition 8.13 and Proposition 8.15   in \cite{Em}
   prove the following  lemma by replacing the tangent bundle  $TM $ with $V$.

   \begin{lemma}\label{le+}  For any $A_0 \in V_x$, there exists a unique   parallel section $A_s$ of $V$ along $X_s$ with initial value $A_0$.   Furthermore,
          \begin{itemize}
\item[i)]    there is a  full measure subset $\Omega_x'$, i.e.,  $\int_{\Omega_x'}\mathcal{W}_x(g)=1 $, such that for any $\gamma \in \Omega_x'$ and all $s\in [0,1]$, the map $\tilde{//}_s(\gamma): V_{x}\rightarrow V_{\gamma(s)}$ given by $A_0 \mapsto A_s(\gamma)$ extends
    the parallel transport $//_s$ along smooth curves, called the stochastic  parallel transport.  Moreover, $\Omega_x' \subseteq C^\beta(\Omega_x)$ for any $0\leq \beta < \frac{1}{2}$.
  \item[ii)] Suppose that  $I$ is an interpolation rule obtained  by Proposition 7.13 in \cite{Em},  $\{s_i=i/m|i=0, 1, \cdots, m\}$ is  a partition  of $[0,1]$,   $X_s^m = I(X_{s_i}, X_{s_{i+1}}, m(s-s_i)), $ for $s\in [s_i, s_{i+1}]$, is the    interpolated process, and  $A_s^m$ is  the parallel section above the piecewise smooth $X_s^m$ with initial condition  $A_0^m=A_0$.  Then  $A_s^m$ converges to  $A_s$ in probability when $m\rightarrow\infty$.  Consequently, $ //_1(X_s^m) $ converges to  $ \tilde{//}_1$  in probability and thus    almost surely by passing to subsequences.     \end{itemize}
    \end{lemma}

  The stochastic holonomy  map is defined as  $$\tilde{//}_1:  \Omega_x' \rightarrow O(r).  $$  It is clear that ${\rm Hol}_x(\nabla)= {\rm Im} (//_1 )\subseteq {\rm Im}( \tilde{//}_1)\subseteq O(r)$, and the assertion ii) implies that $
 \overline{ {\rm Hol}_x(\nabla)}= \overline{{\rm Im}( \tilde{//}_1)}$.

     \begin{definition}[Wiener distribution]\label{def}
         Define a measure  $$ \mu_x(\nabla)=(\tilde{//}_1)_* \mathcal{W}_x(g) $$ on $O(r)$, i.e.,    the push-forward  Wiener measure by   $\tilde{//}_1$.
  \end{definition}

 The measure   $ \mu_x(\nabla)$ depends  on the choices of the connection  $\nabla$, the Riemannian metric  $g$, and the base point $x$ in $M$.

    In the case of $V$ being the trivial complex line bundle on $M$, i.e.,  $V\cong M\times \mathbb{C}$,  a $U(1)$-connection reads
        $\nabla=d+\Gamma$ where $\Gamma$ is a pure imaginary 1-form on $M$ with curvature $d\Gamma$.
    The parallel section $A_s$ along $X_s$
    solves the equation $dA_s+\Gamma\langle \delta X_s \rangle A_s =0$ in the Stratonovich sense, and the solution is
      $A_s=A_0\exp(-\int_0^s \Gamma\langle  \delta  X_t \rangle  ).$
      Thus,  the parallel transport is the exponential of the  Stratonovich integral  $$\tilde{//}_s =\exp(-\int_0^s \Gamma\langle \delta X_t\rangle  ).$$

      Furthermore,    we have a  Stokes type  formula for  general   complex line bundles.

    \begin{proposition}\label{prop+++}  Let $V$ be a complex line bundle on $M$,  $\nabla$ be a $U(1)$-connection, and $F_\nabla$ be the curvature of $\nabla$.  Suppose that
      $\gamma\in \Omega'_x$ satisfies  that  $\tilde{//}_1(\gamma)$ is well-defined, and  $//_1(X_s^m(\gamma))\rightarrow \tilde{//}_1(\gamma)$ when $m\rightarrow\infty$.  If
       $\gamma^o\in \Omega_x^\infty$ is  homotopy to $\gamma$,  then $$\tilde{//}_1(\gamma)=//_1(\gamma^o)\exp(-\int_{\iota(D)}F_\nabla), $$ where
    $\iota: D \rightarrow M$ is a map from the unit disc $D$ in $ \mathbb{C}$ to $M$   such that $\iota$ is smooth in the interior of $D$, and   $\iota|_{\partial D}=\gamma\cdotp (\gamma^o)^{-1}$.
 \end{proposition}

  \begin{proof}[Sketch of  Proof]   It is a direct consequence of $//_1(X_s^m(\gamma))=//_1(\gamma^o)\exp(-\int_{\iota_m(D)}F_\nabla) $, where
    $\iota_m: D \rightarrow M$ is a family of maps  such that $\iota_m|_{\partial D}=X_s^m(\gamma)\cdotp (\gamma^o)^{-1}$.
   \end{proof}

   We remark that the phase factor $\exp(-\int_{\iota(D)}F_\nabla)$ is independent of the choice of $\iota$, since  $$\exp(-\int_{\tilde{\iota}(S^2)}F_\nabla)=\exp (2\pi \sqrt{-1}  \int_{\tilde{\iota}(S^2)}c_1(L))=1  $$   for any map $\tilde{\iota}:S^2\rightarrow M$ whose restriction on the northern  hemisphere  is $\iota$, where $c_1(L)$ is the first Chern class of $L$.

   One unsatisfactory fact about  the  stochastic holonomy map  is that  $\tilde{//}_1$ may   not be continuous. Therefore, $ \mu_x(\nabla)$ is only the push-forward of the Wiener measure by measurable, not necessarily continuous, maps.

  \section{Finite dimensional approximation} Now we introduce  the approximation measures $\mu_x^m(\nabla)$, which are the push-forward measures from  finite dimensional  manifolds by continuous maps.  The construction depends on
   certain  finite dimensional approximations of conditional Wiener measures that bear more resemblance to      those used   in   \cite{DWM}   than the approximation obtained    in  \cite{AD}.
  Generalising  the approximation schemes developed  in  \cite{AD} to  Brownian bridges  might lead to the same approximation measures $\mu_x^m(\nabla)$, which would be left for further studies.

   If inj$(M,g)$ denotes    the  injectivity radius of $(M,g)$,
  then for any two points  $x$ and $y\in M$ with $d_g(x,y)< \rho=\frac{1}{2}$inj$(M,g)$, there is a unique minimal  geodesic connecting $x$ and $y$. For any $m>0$, we consider the partition $\{s_i=i2^{-\nu}| i=0, \cdots, 2^\nu=m\}$ of $[0,1]$. Let $ \Lambda_g^m$ be the open  subset of  $M^{m-1}$ consisting of points $(x_1, \cdots, x_{m-1})\in M^{m-1}$ such that $ d_g(x_i, x_{i+1})<\rho $,     $ i=0, \cdots, m-1$, where $x=x_0=x_m$.   For any $m=2^\nu$,  we define an  embedding $$\Phi_m:  \Lambda_g^m \rightarrow \Omega_x^\infty   $$ such that  $\Phi_m(x_1, \cdots, x_{m-1})$ is the piecewise geodesic given  by the interpolation   $I(x_i, x_{i+1}, m(s-s_i)) $ for $s\in[s_i, s_{i+1}]$,   i.e.,   the unique minimal  geodesic connecting $x_i$ and $ x_{i+1}$.   There are natural embedding  $  \Lambda_g^m \hookrightarrow \Lambda_g^{2m}$ by $(x_1, \cdots, x_{m-1})\mapsto (y_1, \cdots, y_{2m-1})$ where $y_{2i}=x_i$,  $i= 1, \cdots, m-1$, and $y_{2i+1}=I(x_i, x_{i+1}, \frac{1}{2})$. It is clear that  the restriction $$ \Phi_{2m}|_{ \Lambda_g^m}=\Phi_m.$$

   \begin{lemma}\label{le3} For   $m=2^\nu$, the maps $$//_1 \circ \Phi_m:  \Lambda_g^m \rightarrow O(r)$$ are  continuous.  Moreover,  $ \cdots \subseteq \Phi_m( \Lambda_g^m)\subseteq \Phi_{2m}( \Lambda_g^{2m}) \subseteq \cdots \subseteq  \Omega_x^\infty$, and   $$\overline{ \bigcup_{ m\in \mathbb{N}}//_1 \circ\Phi_m( \Lambda_g^m)} =\overline{ {\rm Hol}_x(\nabla)}.$$
      \end{lemma}

  \begin{proof} We embed $M$ in a Euclidean space $\mathbb{R}^N$, and regard a curve $\gamma \in  \Omega_x^\infty$ as a smooth function mapping   from $[0,1]$ to $\mathbb{R}^N$. The first Sobolev norm induces a metric on $\Omega_x^\infty$ given by  $\|\gamma\|_{H^1}^2=\int_0^1|\dot{\gamma}|^2ds$, which is not complete. The stabilities   of the  ordinary differential equations with respect to boundary data and  coefficients imply  that
  $\Phi_m: \Lambda_g^m \rightarrow  \Omega_x^\infty$ and
    the holonomy maps $//_1:\Omega_x^\infty \rightarrow O(r)$  are  continuous.
     Therefore  $//_1 \circ \Phi_m:  \Lambda_g^m \rightarrow O(r)$ are continuous.

       Given a $\gamma\in \Omega_x^\infty$, $|\gamma(s)-\gamma(s')|\leq c_1|s-s'|$ where $c_1\geq |\dot{\gamma}|$, and by taking $|s-s'|< c_1^{-1}\rho$, $\gamma(s)$ and $\gamma(s')$ can be connected by a unique minimal geodesic.
       Thus,   there is a $\gamma^m
     \in \Phi_m(  \Lambda_g^m)$ such that  $\gamma^m(s_i)=\gamma(s_i)$.   It is clear that $\gamma^m$ converges to $\gamma$ in the $H^1$-sense, i.e.
      $\|\gamma-\gamma^m\|_{H^1}^2\rightarrow 0$ when $m\rightarrow\infty$.  Consequently,
      $\Omega_x^\infty =\overline{ \bigcup\limits_{ m\in \mathbb{N}}\Phi_m( \Lambda_g^m)}, $ where the closure is taken with respect to the $H^1$-norm,  and therefore,   $//_1 (\overline{ \bigcup_{ m\in \mathbb{N}}\Phi_m( \Lambda_g^m)}) = {\rm Hol}_x(\nabla).$
      We obtain the conclusion by the continuity of  $//_1$.
  \end{proof}

  \begin{lemma}\label{le4} For any $\varepsilon >0$, there is an $m_0>1$ such that when $m>m_0$, $$1-\varepsilon \leq \frac{1}{p_1(x,x)}\int_{\Lambda_g^m} \prod_{i=0}^{m-1} p_{\frac{1}{m}}(x_{i}, x_{i+1})\prod_{i=1}^{m-1}dv_g(x_i) \leq 1$$ where $x=x_0=x_m$.
      \end{lemma}
   \begin{proof}
If $(x_1, \cdots, x_{m-1})\in M^{m-1}\backslash \Lambda_g^m$, then there exists a $j$, $0\leq j\leq m-1$, such that  $d_g(x_j, x_{j+1})\geq \rho$, and  $$p_{\frac{1}{m}}(x_j, x_{j+1})\leq 2 (\frac{m}{4\pi})^{\frac{n}{2}}e^{-\frac{m\rho^2}{4}} ,$$ for $m\gg 1$.
We calculate   \begin{eqnarray*} & &\int_{ M^{m-1}\backslash \Lambda_g^m}\prod_{i=0}^{m-1} p_{\frac{1}{m}}(x_{i}, x_{i+1})\prod_{i=1}^{m-1}dv_g(x_i)\\  & \leq & 2(\frac{m}{4\pi})^{\frac{n}{2}}e^{-\frac{m\rho^2}{4}} \int_{ M^{m-1}} \prod_{i=0, i\neq j}^{m-1} p_{\frac{1}{m}}(x_{i}, x_{i+1})\prod_{i=1}^{m-1}dv_g(x_i) \\ & = & 2(\frac{m}{4\pi})^{\frac{n}{2}}e^{-\frac{m\rho^2}{4}}  \int_{ M} p_{\frac{j}{m}}(x, x_{j})dv_g(x_j) \int_{ M} p_{1-\frac{j+1}{m}}(x, x_{j+1})dv_g(x_{j+1})\\ &\leq & 2 (\frac{m}{4\pi})^{\frac{n}{2}}e^{-\frac{m\rho^2}{4}}  \rightarrow 0, \end{eqnarray*}
when $m\rightarrow\infty$.        The conclusion follows from      $$\int_{M^{m-1}} \prod_{i=0}^{m-1} p_{\frac{1}{m}}(x_{i}, x_{i+1})\prod_{i=1}^{m-1}dv_g(x_i)=p_1(x,x). $$
 \end{proof}

 Heuristically, $$(\Lambda_g^m,  \frac{1}{p_1(x,x)}\prod_{i=0}^{m-1} p_{\frac{1}{m}}(x_{i}, x_{i+1})\prod_{i=1}^{m-1}dv_g(x_i)) '\rightarrow' (\Omega_x, \mathcal{W}_x(g))$$ when $m\rightarrow\infty$, i.e.,  the Wiener measure $\mathcal{W}_x(g)$ is approximated by the  measures on the $(m-1)$-dimensional manifolds $\Lambda_g^m$.

  \begin{definition}[Approximation measure]\label{def+}
 Define a measure  $$  \mu_x^m(\nabla)=(//_1 \circ \Phi_m)_* \frac{1}{p_1(x,x)} \prod_{i=0}^{m-1} p_{\frac{1}{m}}(x_{i}, x_{i+1})\prod_{i=1}^{m-1}dv_g(x_i) $$ on $O(r)$.
    \end{definition}

   Since  $//_1 \circ \Phi_m$ is continuous, $ {\rm supp} (\mu_x^m(\nabla))=\overline{//_1 \circ\Phi_m( \Lambda_g^m)}$, and
    \begin{equation}\label{eq-supp}\overline{ \bigcup_{ m\in \mathbb{N}}{\rm supp} (\mu_x^m(\nabla))} =\overline{ {\rm Hol}_x(\nabla)}\end{equation} by Lemma \ref{le3}.

      \begin{theorem}\label{thm+2}  When $m\rightarrow\infty$, $ \mu_x^m(\nabla)$ converges to $ \mu_x(\nabla)$ on $O(r)$ in the distribution sense.
       \end{theorem}

        \begin{proof} We regard the piecewise smooth process $X_s^m$ as a measure map $\Xi^m: \Omega_x \rightarrow M\times\cdots\times M=M^{m-1}$ given by $\gamma \mapsto (\gamma(\frac{1}{m}), \cdots, \gamma(\frac{m-1}{m}))$. The distribution of  the Wiener measure is $$(\Xi^m)_* \mathcal{W}_x(g) =\frac{1}{p_1(x,x)} \prod_{i=0}^{m-1} p_{\frac{1}{m}}(x_{i}, x_{i+1})\prod_{i=1}^{m-1}dv_g(x_i)  $$  by the definition of Wiener measure. Let $\tilde{\Phi}_m: M^{m-1} \rightarrow \Omega_x^\infty$ be defined by the interpolation $\tilde{\Phi}_m(x_1, \cdots, x_{m-1})(s)=I(x_i, x_{i+1}, m(s-s_i))$ for $s\in[s_i, s_{i+1}]$.
          Note that   $\tilde{\Phi}_m|_{\Lambda_g^m}= \Phi_m$ and  $\tilde{\Phi}_m\circ \Xi^m(\gamma) (s)=X^m_s(\gamma)$.
         The assertion ii) of Lemma \ref{le+} shows that  $//_1\circ \tilde{\Phi}_m\circ \Xi^m \rightarrow \tilde{//}_1$ in probability, which implies that  $//_1\circ \tilde{\Phi}_m\circ \Xi^m \rightharpoonup \tilde{//}_1$ in distribution (cf. Theorem 1.32 of \cite{ELE}).  Consequently, $$(//_1)_*(\tilde{\Phi}_m)_*(\Xi^m)_* \mathcal{W}_x(g) \rightharpoonup (\tilde{//}_1)_* \mathcal{W}_x(g)= \mu_x(\nabla)$$ in the distribution sense when $m \rightarrow\infty$.   By Lemma \ref{le4},   $$\int_{ M^{m-1}\backslash \Lambda_g^m}(\Xi^m)_* \mathcal{W}_x(g) \rightarrow 0,  \  \  \  \ m\rightarrow\infty,   $$ and we obtain the conclusion.
          \end{proof}

\section{Proofs}

   \begin{proof}[Proof of Theorem \ref{main}]  To prove this theorem, we   modify the induction argument in the proof of   Proposition 8.13    in \cite{Em}.  More precisely,  we  follow the arguments  closely  and  replace    the existence result  of solutions of  stochastic differential equations used   in the proof  of   Proposition 8.13    in \cite{Em}  by the stability theorems (cf. Theorem 15 of  Chapter V  in  \cite{P}, or  Section 5.3 in \cite{Evans}).

    Let $M$, $V$,  $g$, $h^k$, $h$,    $\nabla^k$, and $\nabla$ be the same as in the hypotheses.
     Let $A_s^k$ and $A_s$ be the solutions to  $\nabla^k A_s^k \langle \delta X_s\rangle  =0=\nabla A_s \langle \delta X_s\rangle$ respectively such that  the  initial vectors  $A^k_0$ and $A_0 $ in $V_x$ satisfy   $A^k_0 \rightarrow A_0$ in $ V_x $ as $k\rightarrow\infty$.
         Let
$ \mathcal{U}=\{U \}$
   be a coordinate open covering      of $M$,  and $0=S_0 \leq  \cdots  \leq S_m \leq \cdots$ be
     a sequence of stopping times  such that  for any $j$ and  any  $(\gamma, s)\in \mathcal{I}_j= [[S_j, S_{j+1}]]\cap \{(\gamma, s) \in
   \Omega_x\times [0,1] |S_j(\gamma)  < S_{j+1}(\gamma) \}$, we have $X_s(\gamma)\in U $ for a certain $U\in \mathcal{U}$.
   Furthermore, the collection $\{\mathcal{I}_j\} $ covers $ \Omega_x\times [0,1] $.

  Since $\nabla^k$ converges to  $\nabla$, $\Gamma^k$ converges locally  to $\Gamma$ in the $C^0$-sense on each  $U \in \mathcal{U}$.  There is a  coordinate  chart  $ U_0\in \mathcal{U}$ such that $x\in U_0$, and $X_s(\gamma)\in U_0 $ for any  $(\gamma, s)\in \mathcal{I}_0 $. On $U_0$, there is a trivialisation $V|_{U_0}\cong U_0\times \mathbb{R}^r $.   If  $A_s^{k|S_{1}}$ solves the   equations (\ref{eq0+}) with   $ A_0^{k|S_{1}}=A^k_0$   on $U_0$ in the Itô sense, then $A^{k|S_{1}}$ converges to $A^{|S_{1}}$ in probability, i.e.  $$ 0\leq  \mathbb{E}\big(\sup_{0\leq s \leq S_1}|A_s^{k|S_{1}}-A_s^{|S_{1}}|\big) \rightarrow 0,$$ when $k\rightarrow \infty$ by $ A_0^{k|S_{1}} \rightarrow A_0^{|S_{1}}=A_0$, and  the stability of stochastic  differential equation with respect to the coefficients (cf. Theorem 15 of  Chapter V  in  \cite{P}).

  We use  mathematical induction.       Assume that the stopped $A_s^{k|S_j}$ converges to $A_s^{|S_j}$ in probability  when $k\rightarrow\infty$, which are  driven by the stopped  $X_s^{|S_j}$.
   We consider the $X'_s=X_{\min \{S_j+s, S_{j+1}\}} $ under the conditional probability $\mathbb{P}_g(\cdot|\{\gamma \in
   \Omega_x |S_j(\gamma) < S_{j+1}(\gamma) \})$,     and
    apply  Theorem 15 of  Chapter V in  \cite{P}    to the equation (\ref{eq0+}).  More precisely, if  $A'^{k}_s=A^k_{\min \{S_j+s, S_{j+1}\}}$, then $A'^{k}_s$ (respectively $A'_s$) solves the equation (\ref{eq0+}) with coefficients $\Gamma^k$ (respectively $\Gamma$), and
       $A'^{k}_0=q A^{k|S_j}_{S_j} $ where    $q: U_j \cap U_{j+1} \rightarrow {\rm Hom} (\mathbb{R}^r, \mathbb{R}^r)$ is the transition function between two trivialisations    of the bundle $V$.  The induction hypothesis  shows that  $A'^{k}_0$ converges to $A'_0$ in probability.
    Thus  $A'^{k}_s \rightarrow A'_s $ in probability, and    $A^{k|S_{j+1}}$ converges to $A^{|S_{j+1}}$ in probability by (8.10) in \cite{Em}. Consequently
         $$ 0\leq  \mathbb{E}(|A_1^k-A_1|)\leq  \mathbb{E}\big(\sup_{s \in [0 , 1]}|A_s^k-A_s|_h\big) \rightarrow 0,$$ when $k\rightarrow \infty$, by (1.3) of  \cite{Em} via the induction of $j$.

          We  take an orthonormal basis $ A_{0,1}^k \cdots,  A_{0,r}^k \in V_x$ (respectively $ A_{0,1} \cdots,  A_{0,r} \in V_x$) with  respect to $h^k$ (respectively $h$) such that $ A_{0,j}^k \rightarrow  A_{0,j}$, $j=1, \cdots, r$, when $k\rightarrow\infty$. If we identify $V_x$ with $\mathbb{R}^r$ via $ A_{0,1}^k \cdots,  A_{0,r}^k$, then $\tilde{//}_1^k$  is given by the matrix  $ [A_{1,1}^k \cdots,  A_{1,r}^k]$.
              Thus we
           obtain  $\mathbb{E}(|A_{1,j}^k-A_{1,j}|)\rightarrow 0$, and   $$\lim_{k\rightarrow \infty}\mathbb{E}(|\tilde{//}_1^k-\tilde{//}_1|)=0, $$ i.e.,  $\tilde{//}_1^k$ converges to $\tilde{//}_1$
     in probability. Here we regards both $\tilde{//}_1^k$ and $\tilde{//}_1$ as elements in Hom$(\mathbb{R}^r, \mathbb{R}^r)$.    By Theorem 1.32 of \cite{ELE},  $\tilde{//}_1^k \rightarrow \tilde{//}_1$
      in distribution, and hence  $(\tilde{//}_1^k)_* \mathcal{W}_x(g)$ converges to $(\tilde{//}_1)_*\mathcal{W}_x(g)$ in the distribution sense on $O(r)$. We arrive at  the conclusion.
   \end{proof}

  \begin{proof}[Proof of Theorem \ref{thm3+}]  If we denote $//_1^k$  as  the holonomy map induced by $\nabla^k$, then $//_1^k$ converges to $//_1$ in the $C^0$-sense by the stability  of ordinary differential  equations with respect to coefficients.  Consequently,  we obtain the assertion i) by the definition of $\mu_x^m(\nabla)$.

  Note  that
     $$ \overline{//_1 \circ \Phi_m( \Lambda_g^m)}= {\rm supp}  ( \mu_x^m(\nabla)),    $$ since $//_1 \circ \Phi_m$ is continuous.
    By  $$\int_{  K} \mu_x^m(\nabla^k) \rightarrow \int_{  K} \mu_x^m(\nabla)=0,$$ we obtain $ {\rm supp}  ( \mu_x^m(\nabla))\subseteq H$, and by Lemma \ref{le3},     $$\overline{ \bigcup_{ m\in \mathbb{N}}{\rm supp}  ( \mu_x^m(\nabla))}=\overline{ \bigcup_{ m\in \mathbb{N}}//_1 \circ\Phi_m( \Lambda_g^m)} \subseteq H.$$  Therefore,
$ {\rm Hol}(\nabla) \subseteq H$.
  \end{proof}

Now, we provide  a concrete example of Theorem \ref{main}.

\begin{example}[Flat $U(1)$-connections]\label{ex1}
Let $M$ be
  a non-simply connected manifold,   $x\in M$,  and  $V= M\times \mathbb{C}$ be the trivial complex line  bundle.  Note that  $$ \Omega_x=\coprod_{\nu \in    \pi_1( M)} \Omega_x(\nu), \  \  \  \  \ \sum\limits_{\nu \in    \pi_1( M)} \int_{\Omega_x(\nu)}\mathcal{W}_x(g) =1,  \  \  \  \  \int_{\Omega_x(\nu)}\mathcal{W}_x(g)>0,   $$ where $\Omega_x(\nu)$ are connected components.  The moduli  space of flat $U(1)$-connections on $V$ is the Jacobian torus  $H^1(M, \mathbb{R}/\mathbb{Z})$, and the trivial connection $\nabla^0=d$ is represented by $0\in H^1(M, \mathbb{R}/\mathbb{Z})$.

   A flat  connection $\nabla=d+\Gamma$ is given by a closed pure imaginary  1-form $\Gamma$.  If $\gamma\in \Omega_x$ is a smooth loop, then  the parallel transport reads  $//_s(\gamma)=\exp(-\int_0^s \Gamma\langle \dot{\gamma}\rangle ds )$.  If $\gamma'$ is another loop homotopic to $\gamma$, then  $\int_\gamma \Gamma = \pm\int_{\gamma'} \Gamma$ by $d\Gamma=0$. Hence the holonomy map $//_1 : \Omega_x \rightarrow U(1)$ is a local constant map given by $ //_1(\gamma)= \exp(-\int_\gamma \Gamma)$ for any $\gamma\in \Omega_x$.  For an element  $\nu \in    \pi_1( M)$, there is a real number $\theta_\nu \in [0,1]$ such that $//_1(\gamma)=e^{2\pi\sqrt{-1}\theta_\nu}$ by taking any $\gamma \in \Omega_x(\nu)$.  The holonomy group of $\nabla$ is $${\rm Hol}_x(\nabla)=\{e^{2\pi\sqrt{-1}\theta_\nu}\in U(1)| \nu \in    \pi_1( M)\} .$$  If one of $\theta_\nu$ is an irrational number, then ${\rm Hol}_x(\nabla)$ is dense in $U(1)$, and if all   $\theta_\nu$ are rational numbers, then ${\rm Hol}_x(\nabla)$ is a finite subgroup of  $U(1)$.

    \begin{proposition}\label{pro-flat}   $$ \mu_x(\nabla)= \sum_{\nu \in    \pi_1( M)} \delta_{e^{2\pi\sqrt{-1}\theta_\nu}} \int_{\Omega_x(\nu)} \mathcal{W}_x(g),  \  \  \  \  \  \  \   \mu_x(\nabla^0)=\delta_1 $$ on $U(1)$, where $\delta_z$ is the Dirac delta function on $U(1)$. Consequently,  the support of $\mu_x(\nabla)$   $$ {\rm supp} (\mu_x(\nabla))= \overline{ {\rm Hol}_x(\nabla)}. $$ 
 \end{proposition}

  Let    $ \nabla^k=d+\Gamma^k$ be a family of flat  $U(1)$-connections on $V$ such that  $\|\Gamma^k\|_{C^0}\rightarrow 0$ when $k\rightarrow \infty$, i.e.,  $\nabla^k$ converges to the trivial connection $\nabla^0=d$.
   If $//_s^k(\cdot)$ (respectively $//_s^0(\cdot)$) denotes the parallel transport  induced by $\nabla^k $ (respectively $\nabla^0$), then $//_1^k(\gamma)=\exp (-\int_\gamma \Gamma^k) $ and $//_s^0=$Id.
Note  that
   ${\rm Hol}_x(\nabla^0)=\{1\}$, and
    ${\rm Hol}_x(\nabla^k)$ are generated by $e^{2\pi\sqrt{-1}\theta^k_{\nu_j}}$, $j=1, \cdots, l$, where $\nu_j$ generate the fundamental group $\pi_1(M)$, and $ 2\pi\sqrt{-1}\theta^k_{\nu_j}=-\int_{\gamma_j}\Gamma^k \rightarrow 0$,  $[\gamma_j]=\nu_j$, when $k\rightarrow\infty$.  Therefore,  ${\rm Hol}_x(\nabla^k)$ become increasingly dense  in $U(1)$  when $k\rightarrow\infty$, leading to   $${\rm Hol}_x(\nabla^k) \rightarrow U(1), $$ in the Hausdorff sense. The limit is not equal to   ${\rm Hol}_x(\nabla^0)$.

    For any $\varepsilon >0$, there is a sequence of positive integers $\tau_k$ such that $\tau_k \rightarrow\infty$ and  $$\tau_k \max \{\theta^k_{\nu_j} | j=1, \cdots, l\}< \varepsilon\leq  (\tau_k +1)\max \{\theta^k_{\nu_j} | j=1, \cdots, l\}.$$
 If $\nu \in    \pi_1( M)$, then $\nu=\nu_{i_1}^{l_1}\cdots \nu_{i_c}^{l_c}$,   $2\pi \sqrt{-1}\theta_\nu^k= 2\pi \sqrt{-1} (l_1\theta_{\nu_{i_1}}^k+ \cdots + l_c\theta_{\nu_{i_c}}^k) $,  and we denote $|\nu|=\min \{|l_1|+\cdots +|l_c|\}$ where minimum is  taken over  all possible representations  of $\nu$ using  the generators $\nu_1, \cdots, \nu_l$.
  Thus
  $$0 \leq \int_{\{e^{ 2\pi \sqrt{-1}\theta}\in U(1)|\varepsilon < \theta <1-\varepsilon\}} \mu_x(\nabla^k) \leq  \sum_{|\nu| > \tau_{k}, \nu \in    \pi_1( M) }\int_{\Omega_x(\nu)}\mathcal{W}_x(g)  \rightarrow 0,$$  and
    \begin{equation}\label{eq-100}\mu_x(\nabla^k)\rightharpoonup \mu_x(\nabla^0)=\delta_1  \end{equation} in the distribution sense since  $\int_{U(1)} \mu_x(\nabla^k)=1$.
  \end{example}

 \begin{proof}[Proof of Corollary   \ref{pro-lag}] Clearly,  $$\mu_x(\nabla|_{M_t})\rightharpoonup\delta_1,$$ when $t\rightarrow 0$, which implies that
  $\mu_x(\nabla|_{M_0})=\delta_1$ by  $M_t\rightarrow M_0$ in $N$.  Hence   those $\theta_\nu$  in  Proposition \ref{pro-flat} for $\mu_x(\nabla|_{M_0})$ are zero, i.e.,  $\theta_\nu=0$, since the support of the Dirac function $\delta_1$ is $\{1\}\subset U(1)$ and $\int_{\Omega_x(\nu)}\mathcal{W}_x(g)>0$ for any $\nu \in    \pi_1( M)$.
  We obtain that   $\nabla|_{M_0}=d$ is trivial by passing to gauge changes, and  arrive at  the conclusion.
 \end{proof}

\footnotetext[2]{The text is improved by using
 Grammarly.}

\end{document}